\documentclass[12pt]{amsart}
\usepackage{amsfonts, amssymb, amsmath, amsthm, eucal, latexsym, array, pictex}
\usepackage{fullpage}
\usepackage{verbatim}

\input{xy}
\xyoption{all}

\usepackage{color}
\definecolor{darkblue}{rgb}{0,0,.5}
\definecolor{blue-violet}{rgb}{0.54, 0.17, 0.89}
\definecolor{darkviolet}{rgb}{0.58, 0.0, 0.83}
\definecolor{violet-ryb}{rgb}{0.53, 0.0, 0.69}
\definecolor{violet}{rgb}{0.62, 0.0, 1.0}

\usepackage[colorlinks=true,linkcolor=darkblue,urlcolor=darkviolet,citecolor=magenta]{hyperref}

\numberwithin{equation}{section}

\input{cyracc.def}

\font\tencyr=wncyr10 
\font\tencyi=wncyi10 
\font\tencysc=wncysc10 

\def\rus{\tencyr\cyracc}
\def\rusi{\tencyi\cyracc}
\def\rusc{\tencysc\cyracc}

\newtheorem{thm}{Theorem}
\newtheorem{lm}{Lemma} 
\newtheorem{cl}[lm]{Corollary}
\newtheorem{prop}[lm]{Proposition}

\theoremstyle{remark}
\newtheorem{ex}[lm]{Example}
\newtheorem*{spec}{Speculation}
\newtheorem*{rem}{Remark}

\theoremstyle{definition}
\newtheorem{df}[lm]{Definition}

\newcommand {\eus}{\EuScript}

\newcommand {\bb}{{\boldsymbol{b}}}
\newcommand {\bF}{{\boldsymbol{F}}}
\newcommand{\gt}{\mathfrak}

\newcommand{\ind}{{\rm ind\,}}

\newcommand{\rk}{\mathrm{rk\,}}

\newcommand{\Lie}{\mathrm{Lie\,}}

\newcommand{\gr}{\mathrm{gr\,}}

\newcommand{\Sym}{\mathrm{symm}}
\newcommand{\ad}{\mathrm{ad}}

\newcommand {\ca}{{\mathcal A}}

\newcommand {\gc}{{\mathcal C}}
\newcommand {\cF}{{\mathcal F}}

\newcommand {\cS}{{\mathcal S}}

\newcommand {\U}{{\mathcal U}}

\newcommand {\trdeg}{{\mathrm{tr.deg\,}}}

\newcommand {\mK}{{\mathbb K}}
\newcommand {\mF}{{\mathbb F}}

\renewcommand{\le}{\leqslant}
\renewcommand{\ge}{\geqslant}

\newfam\eusfam%
\font\euszw=eusm10 scaled 1200%
\font\eusac=eusm7 scaled 1200%
\font\eusacc=eusm7 scaled 1000%
\textfont\eusfam=\euszw\scriptfont\eusfam=\eusac%
\scriptscriptfont\eusfam=\eusacc%
\begin{document}
\hfill {\scriptsize January 30, 2020} 
\vskip1ex

\title[Commutative subalgebras of $\U(\gt q)$]
{Commutative subalgebras of $\U(\gt q)$ of maximal transcendence degree} 
\author{Oksana Yakimova}
\address{Institut f\"ur Mathematik, Friedrich-Schiller-Universit\"at Jena, Jena, 07737, Deutschland}
\email{oksana.yakimova@uni-jena.de}
\thanks{This work   is 
funded by the Deutsche Forschungsgemeinschaft (DFG, German Research Foundation) --- project  number  404144169.}
\maketitle


\section*{Introduction}

Let $\gt q$ be a finite-dimensional 
Lie algebra defined over a field $\mK$ of characteristic zero.  Then the universal enveloping algebra $\U(\gt q)$ is a filtered, associative,
non-commutative (in general) algebra and one may ask a natural question: 
\begin{itemize}
\item[(Q1)]
{\it how large can a commutative subalgebra of\/ $\U(\gt q)$  be? }
\end{itemize}

The symmetric algebra $\cS(\gt q)$ is the associated graded algebra of $\U(\gt q)$ 
and it carries the induced Poisson structure. If ${\mathcal A}\subset \U(\gt q)$ is a commutative algebra, then 
$\gr\!(\mathcal A)\subset\cS(\gt q)$ is a Poisson-commutative subalgebra, i.e.,  the Poisson bracket vanishes on it.  
Basic properties of the  coadjoint representation imply that  in this situation,  
$$
{\rm tr.deg}\,\gr\!(\mathcal A)\le (\ind\gt q+\dim\gt q)/2=:\bb(\gt q).
$$  
For a commutative algebra, the transcendence degree coincides with the Gelfand--Kirillov dimension and using a result of Borho and Kraft \cite[Satz~5.7]{bk-GK}, we obtain 
${\rm tr.deg}\, \mathcal A \le \bb(\gt q)$. 
This leads to a more precise formulation of the first question, 
\begin{itemize}
\item[(Q2)]
{\it is there a commutative subalgebra   ${\mathcal A}\subset\U(\gt q)$ such that 
$\trdeg\ca=\bb(\gt q)$? }
\end{itemize}
Our  main result, Theorem~\ref{thm-U},  asserts that the answer to (Q2) is positive. 

\vskip0.5ex

For   a nilpotent Lie algebra $\gt n$, the existence of a  commutative algebra  $\ca\subset\U(\gt n)$
with $\trdeg\ca=\bb(\gt n)$ is shown in \cite[Lemme~9]{GK}. That algebra $\ca$ plays a r\^ole in the proof of the 
Gelfand--Kirillov conjecture.

In case of a reductive Lie algebra $\gt g=\Lie G$, we have $\ind\gt g=\rk \gt g$ and $\bb(\gt g)$ is equal to the dimension of 
a Borel subalgebra $\gt b\subset\gt g$. 
Take $\gamma\in\gt g^*$ such that $\dim\gt g_\gamma=\rk\gt g$. 
Let $\bar\ca_\gamma\subset\cS(\gt g)$ be the {\it Mishchenko--Fomenko subalgebra} associated with $\gamma$, see Definition~\ref{MF}. Then $\{\bar\ca_\gamma,\bar\ca_\gamma\}=0$ \cite{mf3}
and $\trdeg\bar\ca_\gamma=\bb(\gt g)$ \cite{codim3}. The task of lifting $\bar\ca_\gamma$ to $\U(\gt g)$ is known as 
 {\it Vinberg's quantisation problem}. In full generality it is solved by L.\,Rybnikov \cite{r:si}. 
 The solution produces a commutative subalgebra $\ca_\gamma\subset \U(\gt g)$ such that
$\gr\!(\ca_\gamma)=\bar\ca_\gamma$. Thus, $\trdeg\ca_\gamma=\bb(\gt g)$ and this  
provides  the positive answer to (Q2) in the reductive case.  

\vskip0.5ex

The existence of a Poisson-commutative subalgebra
$\bar\ca\subset\cS(\gt q)$ with $\trdeg\bar\ca=\bb(\gt q)$ was conjectured by Mishchenko and Fomenko \cite{mf}. 
Their conjecture is proved  by Sadetov~\cite{sad}. 
In the proof he used a reduction to the    semisimple case.
The steps of that reduction are clarified in \cite{int}. 
There are two nice isomorphisms of certain invariants, see 
Sections~\ref{sec-i1} and \ref{sec-i2}, which are in the background of 
\cite{sad} and are proven in \cite{int}. 
Using these isomorphisms, we perform the same reduction on the level of 
$\U(\gt q)$.

\vskip0.5ex

Question (Q1) has two immediate generalisations. One can consider commutative 
subalgebras either of quotients of $\U(\gt q)$ or of some natural subrings.  
Both these instances turn out to be intricate. We will address quotients of $\U(\gt q)$ in 
a forthcoming paper. Some observations on commutative subalgebras of 
the invariant ring $\U(\gt q)^{\gt l}$, where $\gt l\subset\gt q$ is a Lie subalgebra, are presented in 
Section~\ref{l-inv}.     

Throughout the paper $\gt g$ stands for a reductive Lie algebra.

\section{Basic facts on Lie and Poisson structures} \label{sec-basic}

The symmetric algebra  $\cS(\gt q)$ is the algebra of regular functions $\mK[\gt q^*]$ on $\gt q^*$. 
For $\gamma\in\gt q^*$, let $\hat \gamma$ be the corresponding skew-symmetric form on $\gt q$ given by
$\hat \gamma(\xi,\eta)=\gamma\big([\xi,\eta]\big)$. Note that the kernel of $\hat\gamma$ is equal to the
stabiliser
\begin{equation} \label{eq-stab}
\gt q_\gamma=\{\xi\in\gt q\mid \ad^*(\xi)\gamma=0\}.
\end{equation}
Let $dF$ denote the differential of $F\in\cS(\gt q)$ and $d_\gamma F$ denote the differential of $F$ at $\gamma\in\gt q^*$. Then $d_\gamma F\in\gt q$. 
A well-known property of the Lie--Poisson bracket on $\cS(\gt q)$ is that 
$$
\{F_1,F_2\}(\gamma)=\hat\gamma(d_\gamma F_1,d_\gamma F_2) \ \ \text{ for all } \ \ F_1,F_2\in\cS(\gt q).
$$ 
The  {\it index} of $\gt q$, as defined by Dixmier,  is the number
\begin{equation} \label{eq-ind}
\ind\gt q=\min_{\gamma\in\gt q^*} \dim\gt q_\gamma
=\dim\gt q-\max_{\gamma\in\gt q^*}\rk\hat\gamma = \dim\gt q-\max_{\gamma\in\gt q^*} \dim(\gt q\gamma),
\end{equation} 
where $\gt q\gamma=\{\ad^*(\xi)\gamma \mid \xi\in\gt q\}$. 
The set of {\it regular\/} elements of $\gt q^*$ is 
$$\gt q^*_{\sf reg}=\{\eta\in\gt q^*\mid \dim \gt q_\eta=\ind\gt q\}.$$ 
Set $\gt q^*_{\sf sing}=\gt q^*\setminus \gt q^*_{\sf reg}$.

Suppose that  $\gt q=\Lie Q$ is an algebraic Lie algebra and $Q$ is a connected affine
algebraic group defined over $\mK$. Then $\dim(\gt q x)=\dim(Qx)$ for each $x\in\gt q^*$. 
By Rosenlicht's theorem, see e.g. \cite[Sect.~2.3]{VP}, we have $\trdeg\mK(\gt q^*)^Q=\ind\gt q$. 

Return to an arbitrary $\gt q$.  
For any subalgebra  $A\subset \cS(\gt q)$ and any $x\in\gt q^*$ set
$$
d_xA=\left<d_x F \mid F\in A\right>_{\mK}\subset T^*_x\gt q^*.
$$
Then 
${\rm tr.deg}\,A=\max\limits_{x\in\gt q^*}\dim d_xA$.
If $A$ is Poisson-commutative, then $\hat x(d_xA,d_xA)=0$ for each $x\in\gt q^*$ and
thereby
\begin{equation} \label{eq-trdeg}
\trdeg A\le \frac{\dim\gt q-\ind\gt q}{2}+\ind\gt q=\bb(\gt q)
\end{equation} 
as mentioned in the Introduction.

For any subalgebra $\gt l\subset \gt q$, let $\cS(\gt q)^{\gt l}$ denote the
{\it Poisson centraliser} of $\gt l$, i.e.,
$$
\cS(\gt q)^{\gt l}=\{F\in\cS(\gt q) \mid \{\xi,F\}=0 \ \text{ for all } \  \xi\in\gt l\}.
$$
The algebra of symmetric invariants $\cS(\gt q)^{\gt q}$ is the {\it Poisson centre} of $\cS(\gt q)$. 
The canonical symmetrisation map $\Sym\!: \cS(\gt q)\to \U(\gt q)$ is an isomorphism of 
$\gt q$-modules. Hence we have an isomorphism of vector spaces 
$\cS(\gt q)^{\gt l}$ and $\U(\gt q)^{\gt l}=\{ u \in\U(\gt q) \mid [u,\gt l]=0\}$ for each $\gt l$. 

According to \cite[Prop.~1.1]{m-y}, 
\begin{equation} \label{eq-trdeg-l}
\trdeg A\le \bb(\gt q)-\bb(\gt l)+\ind\gt l
\end{equation} 
for a Poisson-commutative subalgebra $A\subset\cS(\gt q)^{\gt l}$.   

\begin{df} \label{MF}
For $\gamma\in\gt q^*$,  let $\bar\ca_\gamma\subset\cS(\gt q)$ be the 
 corresponding
{\it Mishchenko--Fomenko subalgebra}, which is generated by all
$\gamma$-{\it shifts} $\partial^{k}_\gamma H$ with $k\ge 0$ of all elements $H\in\cS(\gt q)^{\gt q}$.
\end{df}

\noindent
Note that $\partial^{k}_\gamma H$ is a constant for $k=\deg H$.
We have $\{\bar\ca_\gamma,\bar\ca_\gamma\}=0$ \cite{mf3}.

\subsection{Abelian ideals and their invariants} \label{sec-i1}

Let $\gt h\lhd \gt q$ be an Abelian ideal consisting of $\ad$-nilpotent elements. 
Then $\gt h=\Lie H$, where $H$ is a unipotent algebraic group acting on $\gt q^*$ regularly.  
Since $\U(\gt h)$ is commutative, we have $\U(\gt h)=\cS(\gt h)$.  
Set $\mF=\mK(\gt h^*)$. Then $\gt h\subset \mF$. 
Let $\gt h\otimes_{\gt h}\mF$ be a  one-dimensional vector space over $\mF$
spanned by $\delta=w\otimes\dfrac{1}{w}$ with a non-zero $w\in\gt h$. 
Here $v\otimes 1 = v \delta$ for each $v\in\gt h$. 

\begin{rem} Notation $\gt h\otimes_{\gt h} \mF$ is borrowed from \cite{int}.  
It should be understood in the following way. 
Let us regard $\gt h$ as $\gt h{\cdot} 1$. Then $\gt h\otimes_{\gt h} \mF$ is an $\mF$-vector space spanned by $1\otimes 1$ with the property $v\otimes 1=1\otimes v$ for each $v\in\gt h$. 
\end{rem}


The tensor product  $\gt q\otimes_{\gt h} \mF$  
is an $\mF$-vector space of dimension $\dim\gt q-\dim\gt h+1$  and as such it can be 
identified with $(\gt q/\gt h)\otimes_{\mK}\mF\oplus \mF \delta$. 
Since $\gt h$ is an Abelian ideal, 
$H$ acts on $\mF$ trivially and we have an $\mF$-linear action of $H$ on $\gt q\otimes_{\gt h} \mF$. 
Set $\hat{\gt q}=(\gt q\otimes_{\gt h}\mF)^H$. 
The elements of $\hat{\gt q}$ are linear combinations of elements of $\gt q$ with coefficients from $\mF$. Therefore $\hat{\gt q}$ is a subset of the localised 
enveloping algebra $\U(\gt q)\otimes_{\U(\gt h)}\mF$. Note that $[\xi,w^{-1}]=-w^{-2}[\xi,w]\in\mF\delta$ for 
any $\xi\in\gt q$ and a non-zero $w\in\gt h$. Hence 
$$
[\hat{\gt q},\hat{\gt q}]\subset \gt q\otimes_{\gt h} \mF \subset \U(\gt q)\otimes_{\U(\gt h)}\mF.
$$
Clearly, the commutator of two $H$-invariant elements is again an $H$-invariant. Thus, 
\begin{equation} \label{com}
[\hat{\gt q},\hat{\gt q}]\subset \hat{\gt q}
\end{equation}
and $\hat{\gt q}$ is a finite-dimensional Lie algebra over $\mF$. Furthermore, $\delta\in\hat{\gt q}$.
In view of the fact that $[\hat{\gt q},\gt h]=0$, one can write the Lie bracket of $\hat{\gt q}$ 
in  down to earth terms: 
$$
\left[\sum\limits_{j=1}^{N} c_j \xi_j, \sum\limits_{j=1}^N b_j \eta_j\right]= \sum\limits_{i,j} c_j  b_i[\xi_j,\eta_i]
\ \ \text{ for } \ \ \sum\limits_{j=1}^{N} c_j \xi_j, \sum\limits_{j=1}^N b_j \eta_j\in\hat{\gt q} \   \text{ with } \ 
c_j,b_j\in\mF,\, \xi_j,\eta_j\in\gt q. 
$$
Working over $\mF$, we let $\U(\hat{\gt q})$ stand for the enveloping algebra of 
$\hat{\gt q}$. Then clearly  $\mF\subset \U(\hat{\gt q})$. At the sam time, 
$\delta\in   \U(\hat{\gt q})$ and $\delta\not\in\mF\subset \U(\hat{\gt q})$. 
Let $\U_{\delta}(\hat{\gt q})$ be the subalgebra of $\U(\gt q)\otimes_{\U(\gt h)}\mF$
generated by $\hat{\gt q}$.  Then $\U_{\delta}(\hat{\gt q})\cong \U(\hat{\gt q})/(\delta-1)$
as an associative $\mF$-algebra. 


\begin{ex} \label{ex-s}
Suppose that $\gt q=\gt s\ltimes\gt h$, where $\gt s$ is a
subalgebra of $\gt q$. Then 
 \begin{equation} \label{eq.hat-q}
 \hat{\gt q}=\{\xi\in \gt s\otimes_{\mK}\mF  \mid [v,\xi]=0  \ \forall v\in \gt h\} \oplus \mF\delta, 
\end{equation}
see \cite[Lemma~2.1]{Y-imrn}.
We note that there is an unfortunate misprint in Remark~2.6 in \cite{Y-imrn} and that the Lie bracket on $\hat{\gt q}$, which is defined by 
the inclusion $\hat{\gt q}\subset \U(\gt q)\otimes_{\U(\gt h)}\mF$, is the same as the one extended from $\gt s$. 
\end{ex}

Let $\{\xi_1,\ldots,\xi_m\}$ be a basis for a complement of $\gt h$ in $\gt q$ and 
$\{\eta_1,\ldots,\eta_r\}$ be a basis of $\gt h$. 
Then 
\begin{equation} \label{eq-hatq-lin}
\hat{\gt q}=\mF\delta\oplus \left\{ \sum\limits_{i=1}^{m} c_i \xi_i \mid c_i\in \mF,  \sum\limits_{i=1}^{m} c_i [\xi_i,\eta_j]=0 \ \forall j, 1\le j\le r\right\}, 
\end{equation}
where each $[\xi_i,\eta_j]\in\gt h$ is regarded as an element of $\mF$. 
The rank of the $m{\times}r$-matrix $(\boldsymbol{m}_{ij})$ with $\boldsymbol{m}_{ij}=[\xi_i,\eta_j]$ is equal to 
the dimension of $\gt q\alpha\subset\gt h^*$ for a generic $\alpha\in\gt h^*$. 
Hence
\begin{equation} \label{eq-hatq-dim}
\dim_{\mF} \hat{\gt q}= \dim\gt q-\dim\gt h-\max\limits_{\alpha\in\gt h^*} \dim(\gt q\alpha) + 1= 
\min_{\alpha\in\gt h^*} \dim\gt q_\alpha - \dim\gt h +1. 
\end{equation}
In these terms, $\sum\limits_{i=1}^{m} c_i \xi_i \in\hat{\gt q}$ if and only if 
$\sum\limits_{i=1}^{m} c_i(\alpha)\xi_i\in\gt q_\alpha$ whenever all 
$c_i$ are defined for $\alpha\in\gt h^*$.

Let $\U(\gt q)=\bigcup_{d\ge 0}\U_d(\gt q)$ be the standard  filtration on $\U(\gt q)$. 
Set ${\mathcal W}_d=\U_d(\gt q)\U(\gt h)$. Then clearly $\U(\gt q)=\bigcup_{d\ge 0}{\mathcal W}_d(\gt q)$. 
Assume that ${\mathcal W}_{-1}=0$.
Since $\gt h$ is an Abelian ideal, we have a non-canonical isomorphism of commutative alebras
$$
\gr\!_{\mathcal W} \U(\gt q)= \bigoplus_{d\ge 0} {\mathcal W}_d/ {\mathcal W}_{d-1} \cong \cS(\gt q). 
$$ 
The new filtration extends to $\U(\gt q)\otimes_{\U(\gt h)}\mF$ and on 
$\U_{\delta}(\hat{\gt q})\subset \U(\gt q)\otimes_{\U(\gt h)}\mF$ it coincides with the 
standard  filtration inherited by the quotient  $\U(\hat{\gt q})/(\delta-1)$ from 
$\U(\hat{\gt q})$.

The algebra $\hat{\gt q}$ coincides with the algebra $\widetilde{\gt q}=\widetilde{\gt q}(I_0)$, defined  
in \cite[Sect.~4]{int},  in the particular case $I_0=\{0\}$. 
Therefore $\hat{\gt q}$ is  the quotient of the Lie algebra 
of all rational maps 
$$
\xi\!: \gt h^*\to\gt q \ \text{ such that } \
\xi(\alpha)\in\gt q_\alpha \ \text{ whenever } \ \xi(\alpha) \ \text{ is defined} 
$$
by $\hat{\gt h}:=\{\xi\in\gt h\otimes_{\mK}\mathbb F\mid
\alpha(\xi(\alpha))=0 \mbox{ for each } \alpha\in {\gt h}^*
   \ \text{ such that $\xi(\alpha)$ is defined}\}.$ 

\begin{lm} \label{lm-u-h} {\sf (i)} $(\U(\gt q)\otimes_{\U(\gt h)}\mF)^H=\U_{\delta}(\hat{\gt q})$. \\[.2ex]
{\sf (ii)} $\bb(\hat{\gt q})=\bb(\gt q)-\dim\gt h+1$.   
\end{lm} 
\begin{proof} {\sf (i)} \ 
Set $\cF=(\U(\gt q)\otimes_{\U(\gt h)}\mF)^H$. Note that $\U_{\delta}(\hat{\gt q})\subset \cF$ by the construction. 
Since the action of $H$ on $\gt h^*$ is trivial, we have also 
$\cF=\U(\gt q)^H\otimes_{\U(\gt h)}\mF$.  Because $\hat{\gt q}$ is a Lie algebra over $\mF$, it suffices to show that 
$\U(\gt q)^H\subset \U_{\delta}(\hat{\gt q})$. Employing the filtration $\U(\gt q)=\bigcup_{d\ge 0}{\mathcal W}_d(\gt q)$ and the symmetrisation map 
one readily reduces the claim 
to the   level of $\cS(\gt q)^H$. 

The assertion that $\cS(\gt q)^H\otimes_{\cS(\gt h)}\mF\cong \cS(\hat{\gt q})/(\delta-1)$ is  contained implicitly in 
\cite[Lemma~21]{int}. For the sake of completeness, we briefly recall the argument. 
Set $\boldsymbol{F}= \cS(\gt q)^H\otimes_{\cS(\gt h)}\mF$.

Now we consider $\hat{\gt q}$ as subset of $\bF$ identifying $\delta$ with $1$. 
Both, $\bF$ and the subalgebra 
$\cS_{\delta}(\hat{\gt q})\subset\bF$ generated by $\hat{\gt q}$, are 
vector spaces over $\mK({\gt h}^*)$. 
Thus, it suffices to verify  
the equality $\bF=\cS_{\delta}(\hat{\gt q})$  at generic $\alpha\in{\gt h}^*$.

Let  $Y_\alpha\subset \gt q^*$
be the preimage of $\alpha$ under the canonical  restriction $\gt q^*\to\gt h^*$. 
Since $\gt h$ is commutative, $H$ acts on $Y_\alpha$. 
By \cite[Lemma~20]{int},  $Y_\alpha/H={\rm Spec}(\mK[Y_\alpha]^H)$
and the restriction map
$\pi_\alpha\!:Y_\alpha\to(\gt q_\alpha)^*$ defines an isomorphism
$Y_\alpha/H\cong(\gt q_\alpha/\gt h)^*\times\{\alpha\}$.

Let ${\bF}_{\alpha}\subset\bF$ be the subset of elements  that are defined
at $\alpha$. Then  for any $\alpha\in\gt h^*$, we have a 
map
$$
\epsilon_\alpha\!: {\bF}_{\alpha}\to
  \mK[Y_\alpha]^H\cong\mK[(\gt q_\alpha/\gt h)^*\times\{\alpha\}]\cong\cS(\gt q_\alpha/\gt h).
$$
Eq.~\eqref{eq-hatq-dim} and the discussion after it imply that
$\gt q_\alpha/\gt h$ embeds into  $\epsilon_\alpha(\hat{\gt q}\cap {\bF}_{\alpha})$ 
for a generic point $\alpha$. Hence, if $\alpha$ is generic, then 
$\epsilon_\alpha(\cS_\delta(\hat{\gt q})\cap\bF_\alpha)\cong\cS(\gt q_\alpha/\gt h)$, cf. the proof of  \cite[Lemma~21]{int}. 
Therefore, $\epsilon_\alpha(\bF_\alpha)=\epsilon_\alpha(\cS_\delta(\hat{\gt q})\cap\bF_\alpha)$ and 
we can conclude that  $\bF=\cS_\delta(\hat{\gt q})$. \\[.4ex]
{\sf (ii)} This part is proven in \cite[Sect.~5]{int}. 
Let $\alpha\in\gt h^*$ and $\gamma\in Y_\alpha$ be generic. 
Set  $k=\dim(H\gamma)$.  
Since the form  $\alpha([\,\,,\,])$ defines a non-degenerate pairing 
between $\gt q/\gt q_\alpha$ and $\gt h/\gt h_\gamma$, we have also 
$k=\dim(\gt q \alpha)$.  Note that $\dim_{\mF}\hat{\gt q}=\dim\gt q-k-\dim\gt h+1$ by Eq.~\eqref{eq-hatq-dim}. 

The numerical characteristics of $\hat{\gt q}$, like 
index, can be computed locally, at
$\alpha$, so to say.  In particular, 
$\dim{\gt q}_\alpha-\ind{\gt q}_\alpha=\dim_{\mF}\hat{\gt q}-\ind\hat{\gt q}$.

Write $\gt q=\gt m\oplus(\gt r\oplus\gt h)$, where $\gt r\oplus\gt h=\gt q_\alpha$.
Here $\hat\gamma(\gt q_\alpha,\gt h)=0$ and $\hat\gamma$ is non-degenerate on 
$\gt m\times(\gt h/\gt h_\gamma)$. 
The block structure of $\hat\gamma$ shows that 
$$
\rk\hat\gamma = 2k+ \rk(\hat\gamma|_{\gt r\times\gt r})= (\dim{\gt q}_\alpha-\ind{\gt q}_\alpha)+2k.
$$  
Hence $\dim\gt q-\ind\gt q-2k=\dim\gt q_\alpha-\ind\gt q_\alpha=\dim_{\mF}\hat{\gt q}-\ind\hat{\gt q}$ and 
$$\bb(\hat{\gt q})=\dim_{\mF}\hat{\gt q}-\frac{1}{2}(\dim_{\mF}\hat{\gt q}-\ind\hat{\gt q}) =\frac{1}{2}(\dim\gt q+ \ind\gt q)-\dim\gt h+1.$$ 
This completes the proof.  
\end{proof}


We will need another auxiliary statement. 
Suppose that $z\in\gt q$ is a non-zero central element.  
Let $\ca\in\U(\gt q)$ be a commutative subalgebra. 
For $c\in\mK$, let $\ca(c)$ be the image of $\ca$ in $\U(\gt q)/(z-c)$. 
Then the following assertion is true. 

\begin{lm} \label{delta}
There is a non-zero $c\in\mK$ such that $\trdeg\ca(c)\ge \trdeg\ca -1$.
\end{lm}
\begin{proof}
We consider $\ca$ as a subalgebra of $\U(\gt q)_z=\U(\gt q)\otimes_{\mK[z]}\mK(z)$. 
On $\U(\gt q)_z$, there is an increasing filtration by the finite-dimensional 
$\mK(z)$-vector spaces 
$\U_d(\gt q)_z=\left<\U_d(\gt q)\right>_{\mK(z)}$. The associated graded algebra
$\gr\!_z(\U(\gt q)_z)$ is isomorphic to $\cS(\gt q)_z=\cS(\gt q)\otimes_{\mK[z]}\mK(z)$. Let 
$\bar\ca\subset \cS(\gt q)_z$ be the graded image of $\ca$.
According to \cite[Satz~5.7]{bk-GK}, $\trdeg_{\mK(z)}\bar\ca=\trdeg_{\mK(z)}\ca$. Note that 
actually $\bar\ca\subset\cS(\gt q)$. 

The quotient $\U(\gt q)/(z-c)$ inherits the standard filtration from $\U(\gt q)$. Moreover,  Diagram~\ref{fig-z} is commutative. 
\begin{figure}[h] 
$\xymatrix@R+2mm@C+4mm{ 
\U(\gt q)  \ar[r] \ar[d]^{\gr\!_z}    & \U(\gt q)/(z-c)   \ar[d]^{\gr} \\ 
\cS(\gt q) \ar[r] &  \cS(\gt q)/(z-c)  \\ } 
$
\caption{Filtration of the localised algebra and quotients.} \label{fig-z}
\end{figure}  
Let $\bar\ca(c)$ be the image of $\bar\ca$ in $\cS(\gt q)/(z-c)$. 
One of the basic facts in algebraic geometry states that there is a non-zero $c\in\mK$ such that 
$\trdeg\bar\ca(c)=\trdeg_{\mK(z)}\bar\ca$. 
Hence, for this $c$, we have  
$$
\trdeg\ca(c)\ge \trdeg\bar\ca(c)=\trdeg_{\mK(z)}\bar\ca = \trdeg_{\mK(z)}\ca \ge \trdeg\ca -1
$$
as desired. 
\end{proof}

\subsection{Invariants of a Heisenberg algebra} \label{sec-i2}

Recall that  a $(2n{+}1)$-dimensional Heisenberg Lie algebra over $\mK$
is a Lie algebra $\gt h$ with  a basis
$\{x_1,\ldots,x_n,y_1,\ldots,y_n,z\}$  such that $n\ge 1$,
$[x_i,x_j]=[y_i,y_j]=0$, $[\gt h,z]=0$, and $[x_i,y_j]=\delta_{ij}z$.
Suppose that $\gt q=\gt l\ltimes \gt h$, where $\gt l$ is a subalgebra and $[\gt q, z]=0$. 
Assume further that the subspace $\gt v=\left<x_j,y_j\mid 1\le j\le n\right>_{\mK}$ is $\gt l$-stable. 
According to \cite[Lemma~18]{int} and its corollary,
\begin{equation} \label{eq-hl}
(\cS(\gt q)[z^{-1}])^\gt h \cong \cS(\gt l)\otimes_{\mK} \mK[z,z^{-1}].  
\end{equation}
This isomorphism can be made very explicit. For $\xi\in \gt l$, set 
$$
\hat\xi=\xi +  \frac{1}{2 z}\sum\limits_{i=1}^n ([\xi,x_i] y_i - [\xi,y_i] x_i) \in \U(\gt q)[z^{-1}]. 
$$
The following statement is elementary in nature and is certainly known. Similar ideas have been used in 
\cite[Sect.~4.8]{ppy}. 

\begin{lm}\label{v-com}
We have $[v,\hat\xi]=0$  for all $\xi\in\gt l$ and all $v\in\gt h$.
\end{lm}
\begin{proof}
It is enough to show that  $[x_j,\hat\xi]=[y_j,\hat\xi]=0$  for all $j$ such that $1\le j\le n$. 
We have 
$$
\begin{array}{l}
[x_j ,\hat\xi]=  [x_j, \xi]+ \frac{1}{2z}\left((\sum\limits_{i=1}^n  [x_j,[\xi,x_i]] y_i) + [\xi,x_j]z  - 
  \sum\limits_{i=1}^n  [x_j,[\xi,y_i]] x_i  \right) = \\ 
  \qquad = [x_j, \xi] + \frac{1}{2} [\xi,x_j] +\frac{1}{2z} \left( \sum\limits_{i=1}^n  [[x_j,\xi],x_i] y_i - 
  \sum\limits_{i=1}^n  [[x_j,\xi],y_i] x_i \right)  =  [x_j, \xi]  + [\xi,x_j] =0;  \\ 
{} [y_j,\hat\xi]  = [y_j,\xi]+\frac{1}{2z}\left((\sum\limits_{i=1}^n  [y_j,[\xi,x_i]] y_i) + [\xi,y_j]z  - 
  \sum\limits_{i=1}^n  [y_j,[\xi,y_i]] x_i  \right) =  \\ 
  \qquad = [y_j,\xi] + \frac{1}{2} [\xi,y_j] +\frac{1}{2z} \left( \sum\limits_{i=1}^n  [[y_j,\xi],x_i] y_i - 
  \sum\limits_{i=1}^n  [[y_j,\xi],y_i] x_i \right)  = [y_j,\xi] +[\xi,y_j] =0.  
\end{array}
$$
This completes the proof. 
\end{proof}

It follows from \cite[Lemma~18]{int} that 
$(\cS(\gt q)[z^{-1}])^\gt h$ is generated by the symbols 
$\gr\!(z\hat\xi)$ of the elements $z\hat\xi$ with $\xi\in\gt l$ and by $z,z^{-1}$.  
The same lemma states that \eqref{eq-hl} is a natural isomorphism of 
Poisson algebras. For $\xi,\eta\in\gt l$ and $\zeta=[\xi,\eta]$, 
we  have therefore
\begin{equation} \label{eq-hl-p}
\{\gr\!(z\hat\xi),\gr\!(z\hat\eta)\} = z \gr\!(z\hat\zeta),
\end{equation}
where the Poisson bracket is taken in $\cS(\gt q)$. The commutator 
$[\hat\xi,\hat\eta]$ 
belongs to $(\U(\gt q)[z^{-1}])^{\gt h}$. 

\begin{lm} \label{lm-hat}
In the above notation, we have $[\hat\xi,\hat\eta]=\hat\zeta$.
\end{lm}
\begin{proof}
Write $\hat u=u+T(u)$ for $u\in\gt l$. Then $T(u)\in\U(\gt h)[z^{-1}]$ and 
hence $[\hat v,T(u)]=0$ for all $v\in\gt l$. Now 
$[\hat\xi,\hat\eta]=\zeta+[\xi,T(\eta)]$. 
Next we consider the elements   
$$
\overline{T}(u)=\frac{1}{2}\sum\limits_{i=1}^n ([u,x_i] y_i - [u,y_i] x_i))\in \cS^2(\gt v).
$$ 
By the construction, $\gr\!(z\hat u)=zu+\overline{T}(u)$. 
Since $[\gt l,\gt h]\subset \gt h$,   each $\{\xi,\gr\!(z\hat u)\}$ is  again an $\gt h$-invariant. 
In particular, 
$$
\{\xi, \gr\!(z\hat \eta)\} - \gr\!(z\zeta)=\{\xi,\overline{T}(\eta)\}-\overline{T}(\zeta)\in  \cS^2(\gt v)^{\gt h}.
$$ 
Because $\cS^2(\gt v)^{\gt h}=0$, this difference is zero.  

Observe that 
$$
[[u,x_i],y_i]+[x_i,[u,y_i]]=[u,z]=0
$$
for each $i$ and hence $[[u,x_i],y_i]-[[u,y_i],x_i]=0$. 
This implies that $zT(u)=\Sym(\overline{T}(u))$.  Thereby
$[\xi,zT(\eta)]$ equals the symmetrisation of  $\{\xi,\overline{T}(\eta)\}=\overline{T}(\zeta)$.  
Thus $[\xi,T(\eta)]=T(\zeta)$ and we are done. 
\end{proof}

\begin{cl} \label{cl-heis}
We have $(\U(\gt q)[z^{-1}])^{\gt h}\cong \U(\gt l)\otimes_{\mK}\mK[z,z^{-1}]$. \qed 
\end{cl}

The isomorphism~\eqref{eq-hl} implies that $\ind\gt q=\ind\gt l+1$. 
Hence
\begin{equation} \label{eq-bb-h}
\bb(\gt q)= \bb(\gt l)+n +1. 
\end{equation}

\section{On algebraic extensions} \label{sec-ext}

Let $A=\bigcup\limits_{n\ge 0} A_n$ be an increasingly  filtered associative algebra such that $\dim A_n<\infty$ for each $n\ge 0$. 
Assume that $A_{-m}=0$ for all $m\ge 1$.
Suppose that 
the 
associated graded algebra $\overline{A}=\gr A$ is a commutative domain
and a finitely generated $\mK$-algebra.   
For each $a\in A_n\setminus A_{n-1}$, set $\bar a=\gr\!(a)=a+A_{n-1}$. 
For a subspace $V\subset A$, let $\overline{V}=\gr\!(V)$ be the subspace of 
$\overline{A}$ spanned  by  $\gr\!(v)$ with $v\in V$. 

Let $B$ be a subalgebra of $A$. 
Then \cite[Satz~5.7]{bk-GK} asserts that 
the Gelfand--Kirillov dimensions of $B$ and $\overline{B}$ are equal.  
This result implies that for commutative subalgebras $B\subset C\subset A$, 
 we have 
\begin{equation} \label{eq-A}
C \ \text{ is algebraic over } \ B \ \Longleftrightarrow \ 
  \gr\!(C) \ \text{ is algebraic over } \ \gr\!(B).
\end{equation}

\begin{lm}[{cf. \cite[Lemma~1]{r:wc}}] \label{lm-closure}
Keep the above assumptions on $A$ and  let $B\subset C$ be commutative subalgebras
of $A$ such that  $C$ is algebraic over $B$. 
If $[x,B]=0$ for some $x\in A$, then also  $[x,C]=0$. 
\end{lm}
\begin{proof}
Since $\overline{A}$ is commutative, we have $[A_n,A_m]\subset A_{n+m-1}$ for 
all $m,n\ge 0$. Assume that there is $x\in A_m\setminus A_{m-1}$ such that
$[x,B]=0$ and  $[x,C]\ne 0$.  
Let $k\ge 1$ be the minimal number  such that
\begin{equation} \label{eq-k}
[x,C\cap A_n]\subset A_{m+n-k} \ \text{ for all } \ n\ge 0,
\end{equation}
but 
$[x,C\cap A_n]\not\subset A_{m+n-k-1}$ for some $n$. 
For $\bar u\in \overline{A}_\ell$ with $u \in C$, set 
$$
 \{\bar x,\bar u\}_k =  [x,u] + A_{m+\ell-k-1}.
$$
If $\bar u=\gr\!(u')$ with $u'\in C$, then $u-u'\in (C\cap A_{\ell-1})$ and hence 
$[x,u-u']\in A_{m+\ell-k-1}$ because of \eqref{eq-k}. 
Thereby $\{\bar x,y\}_k$ is a well-defined element of $A_{m+\ell-k}/A_{m+\ell-k-1}$ for each 
$y\in (\overline{C}\cap \bar A_\ell)$. The linear map 
$\{\bar x,\,\}\!: \gr\!(C)\to \overline{A}$  
satisfies the Leibniz rule by the construction 
and 
$\{\bar x,\overline{B}\}=0$. 

There is $u\in(C\cap A_n)$ such that
$\{\bar x,\bar u\}_k\ne 0$. Since $u$ is algebraic over $B$, the symbol $\bar u$ is algebraic over $\overline{B}$. 
Let 
$$
\boldsymbol{Q}(\bar u)=\bar b_N \bar u^N + \ldots + \bar b_1 \bar u + \bar b_0=0
$$ 
with $b_j\in B$ 
be a non-trivial equation on $\bar u$ of the smallest possible degree.  
Since $\bar u$ is a homogeneous element of $\overline{A}$, we can assume that all 
summands $\bar b_j \bar u^j$ have one and the same degree in $\overline{A}$. 

Consider the  symbol of $[x,\widetilde{\boldsymbol{Q}}(u)]$, where $\,\widetilde{\boldsymbol{Q}}(X)=\sum b_j X^j$. 
This symbol  is equal to the product 
$$
\{\bar x,\bar u\}_k ( N  \bar b_N \bar u^{N-1}+\ldots + 2\bar b_2 \bar u+ \bar b_1)=\{\bar x,\boldsymbol{Q}(\bar u)\}_k=0. 
$$
Because $\overline{A}$ is a domain, we have obtained an equation on 
$\bar u$ of smaller than $N$ degree. This contradiction proves that $[x,C]=0$ whenever $[x,B]=0$.  
\end{proof}

\begin{cl} \label{cl-cl}
Let $B\subset A$ be as in Lemma~\ref{lm-closure}. Then the algebraic closure of $B$
in the centraliser $Z_A(B)\subset A$ 
is a commutative subalgebra. \qed  
\end{cl}

\begin{rem}
A well-known fact is that the algebraic closure of a Poisson-commutative subalgebra is again Poisson-commutative. 
Lemma~\ref{lm-closure}, which is inspired by 
\cite[Lemma~1]{r:wc}, can be regarded as a non-commutative 
generalisation of this statement. 
\end{rem}

Our main example of $A$ is $\U(\gt q)$. Here $\gr A=\cS(\gt q)$ is a finitely  generated
commutative algebra, which is a domain.

\section{The inductive argument} \label{sec-ind}

Let $\gt n\lhd\gt q$ be the nilpotent radical of $\gt q$. Note that 
$\gt n$ is an algebraic Lie algebra. 

\begin{lm}[{ \cite[Lemma~4.6.2]{Dix}, cf. \cite[Lemma~17]{int} }]  \label{lm-h}
Suppose that each commutative non-zero characteristic
ideal of $\gt n$ is one-dimensional and $\gt n\ne 0$. Then either $\gt n=\mK$\/ or $\gt n$ is a Heisenberg Lie algebra. \qed 
\end{lm}

\begin{rem}
It is a borderline issue, whether to consider $\mK$ as a Heisenberg Lie algebra. 
In \cite[Lemma~17]{int}, the convention is that $\mK$ is included into the class of Heisenberg algebras. 
Note that the results of \cite[Section~4]{int} are valid for $\mK$ as well by a  trivial reason. 
\end{rem}

An algebraic Lie algebra $\gt q$ has an algebraic Levi decomposition $\gt q=\gt l\ltimes\gt n$, where $\gt l$ is reductive. In the non-algebraic case, $\gt q=\gt s\ltimes\gt r$, where $\gt r$ is the solvable radical of 
$\gt q$ and $\gt s$ is semisimple. As is well-known, $[\gt r,\gt r]\subset \gt n$. 
Moreover, $\gt n\ne 0$ in the non-algebraic case, because otherwise $\gt q$ 
were reductive. 
The case  of a non-algebraic $\gt q$ is more involved and requires an additional lemma. 
 
\begin{lm} \label{non}
Let $\gt h=\gt v\oplus\gt z$ be a Heisenberg Lie algebra, where $\gt z=\mK z$ is the centre 
of $\gt n$ and 
$\dim\gt v\ge 2$.  
Suppose that $\gt h$ is an ideal of $\gt q$.  
Set $\widetilde{\gt l}=\{\xi\in\gt q \mid [\xi,\gt v]\subset \gt v\}$. Then $\gt q=\widetilde{\gt l}+\gt h$ and 
$\widetilde{\gt l}\cap\gt h=\gt z$.  
\end{lm}
\begin{proof} 
The equality $\widetilde{\gt l}\cap\gt h=\gt z$ follows from the structure of $\gt h$. 
Take any $\xi\in\gt q$. Then $\ad(\xi)$ defines a linear map from $\gt v$ to $\gt h/\gt v\cong\gt z$. 
Any such map can be presented as a commutator with some $\eta\in\gt v$. 
Hence there is $v\in\gt v$ such that 
$\ad(\xi)-\ad(v)$ preserves $\gt v$. 
Here $\xi-v\in\widetilde{\gt l}$ and we are done. 
\end{proof}
 
The construction of Section~\ref{sec-i2} generalises easily to the non-algebraic setting leading to the following statement. 
 
 \begin{cl} \label{heis-non} Keep the assumptions and notation of Lemma~\ref{non} and 
suppose additionally that $[\gt q,\gt z]=0$. Then 
$(\U(\gt q)[z^{-1}])^{\gt h} \cong\U(\widetilde{\gt l})[z^{-1}]$.  \qed 
\end{cl}
 
One more observation is required before we can start the induction. 

\begin{lm} \label{field}
In the reductive case, quantum MF-subalgebras $\ca_\gamma$ are defined over $\mathbb Q$ and hence over any field of characteristic zero. 
If $\gamma\in\gt g^*_{\sf reg}$, then $\trdeg\ca_\gamma=\bb(\gt g)$. 
\end{lm} 
\begin{proof}
Recall the construction from \cite{r:si}. 
Let $G$ be a connected {\bf complex} reductive algebraic group. Set $\gt g=\Lie G$. 
The universal enveloping algebra $\U\big(t^{-1}\gt g[t^{-1}]\big)$
contains a certain commutative subalgebra $\gt z(\widehat{\gt g})$, which is known as 
the {\it Feigin--Frenkel centre}. 
Set $\gt l=[\gt g,\gt g]$, $r=\dim\gt l$.
According to \cite{r:un}, $\gt z(\widehat{\gt g})$  
is the centraliser in $\U\big(t^{-1}\gt g[t^{-1}]\big)$ of the following 
quadratic element 
$$
{\mathcal H}[-1]=\sum_{a=1}^{r} x_a t^{-1} x_a t^{-1}, 
$$
where  $\{x_1,\ldots,x_{r}\}$ is any basis of $\gt l$ that is 
orthonormal w.r.t. the Killing form. 

For any $\gamma\in\gt g^*$ and
a non-zero $z\in\mathbb C$, the map 
\begin{equation} \label{ff-map}
\varrho_{\gamma,z}\!:\U\big(t^{-1}\gt g[t^{-1}]\big)\to \U(\gt g),
\qquad x t^r \mapsto  z^r x +\delta_{r,-1}\gamma(x),\quad x\in\gt g,
\end{equation}
defines  a $G_\gamma$-equivariant  algebra homomorphism. The image of $\gt z(\widehat{\gt g})$ 
under $\varrho_{\gamma,z}$
is a commutative subalgebra $\ca_{\gamma}$ of $\U(\gt g)$, which does not depend
on $z$ \cite{r:si}. Moreover, $\bar\ca_\gamma\subset \gr\!(\ca_\gamma)$ for each $\gamma\in\gt g^*$ \cite{r:si}.
If $\gamma\in\gt g^*_{\sf reg}$, then $\bar\ca_\gamma$ is a maximal w.r.t. inclusion Poisson-commutative 
subalgebra of $\cS(\gt g)$ \cite{codim3} and hence 
$\gr\!(\ca_\gamma)=\bar\ca_\gamma$. 

If $\gt l$ is simple, then $H[-1]=\gr\!({\mathcal H}[-1])$ spans 
$\cS^2(\gt l t^{-1})^{\gt g}$. In general, $\gt z(\widehat{\gt g})$  is the centraliser 
of $\Sym(\cS^2(\gt g t^{-1})^{\gt g})$. The subspace $\cS^2(\gt g t^{-1})^{\gt g}$ has a $\mathbb Q$-form and 
behaves well under field extensions. Its centraliser in $\U\big(t^{-1}\gt g[t^{-1}]\big)$ shares the
same properties. 

If $\gt g$ is a Lie algebra over $\mathbb K$ and $\mathbb K\subset\mathbb L$, then 
$$
\cS_{\mathbb L}^2(\gt g(\mathbb L) t^{-1})^{\gt g(\mathbb L)} = \cS^2(\gt g t^{-1})^{\gt g}\otimes_{\mK}\mathbb L
$$
for $\gt g(\mathbb L)=\gt g\otimes_{\mK}\mathbb L$. If $\gamma\in\gt g^*$ and $\gamma(\mathbb L)\in\gt g(\mathbb L)^*$ 
is its continuation, then 
$\bar\ca_{\gamma(\mathbb L)}=\bar\ca_\gamma\otimes_{\mK}\mathbb L$.  
Playing with extensions 
$\mathbb Q\subset \overline{\mathbb Q}\subset\mathbb C$ and $\mK\subset\overline{\mK}$, 
one shows that $\gt z(\widehat{\gt g})$ produces a quantum MF-subalgebra over any $\mK$.
In more details, since $\bar\ca_\gamma\subset \gr\!(\ca_\gamma)$ holds over $\mathbb C$, it holds over $\mathbb Q$ and 
$\overline{K}$, thereby it holds over $\mathbb K$. By \cite{codim3}, 
$\trdeg\bar\ca_\gamma =\bb(\gt g)$ for $\gamma\in\gt g^*_{\sf reg}$ over $\overline{\mK}$. Hence also 
$\trdeg\bar\ca_\gamma =\bb(\gt g)$ for $\gamma\in\gt g^*_{\sf reg}$ over $\mK$ and 
$\trdeg\ca_\gamma=\bb(\gt g)$ over $\mK$. 
\end{proof}
 
\begin{thm} \label{thm-U}
For each finite-dimensional Lie algebra $\gt q$, there is
a commutative algebra $\ca\subset\U(\gt q)$ with $\trdeg\ca=\bb(\gt q)$. 
\end{thm}
\begin{proof}
There is no harm in assuming that $\gt q$ is indecomposable. 
The case of a simple (reductive) Lie algebra $\gt g$ is settled by a result of Rybnikov \cite{r:si}, 
here  $\trdeg\ca_\gamma=\bb(\gt g)$ for a quantum Mishchenko--Fomenko subalgebra
$\ca_\gamma$, see also Lemma~\ref{field} and the Introduction.   
Therefore suppose that  $\gt n\ne 0$. 
In this case we argue by induction on $\dim\gt q$. The induction begins with $\dim\gt q=1$, where $\bb(\gt q)=1$ 
and there is nothing to prove. 

$\bullet$ \ Suppose first  that there is an Abelian Ideal $\gt h\lhd\gt q$ such that $\gt h\subset\gt n$ and 
$[\gt q,\gt h]\ne 0$ or $\dim\gt h>1$. Let $H$, $\mF$, and $\hat{\gt q}$ be the same as in Section~\ref{sec-i1}. 
We have 
$$
\dim_{\mF}\hat{\gt q} \le \dim_{\mF}(\gt q\otimes_{\gt h}\mF) = \dim_{\mK}\gt q-\dim\gt h+1.
$$
Moreover, $\dim_{\mF}\hat{\gt q} < \dim_{\mF}(\gt q\otimes_{\gt h}\mF)$ if $[\gt q,\gt h]\ne 0$. 
By the assumptions on $\gt h$,  $\dim_{\mF}\hat{\gt q}< \dim_{\mK}\gt q$. 
By the inductive hypothesis, $\U(\hat{\gt q})$ contains a commutative subalgebra  
$\widetilde{\ca}_1$ such that $\trdeg\!_{\mF}\,\widetilde{\ca}_1=\bb(\hat{\gt q})$. 
Without loss of generality, assume that $\ca_1$ contains the central element $\delta\in\hat{\gt q}$.  By Lemma~\ref{delta}, there is a non-zero $c\in\mF$ such that 
$\trdeg\!_{\mF}\,\ca_1=\bb(\hat{\gt q})-1$ for the image $\ca_1=\widetilde{\ca}_1(c)$ of $\widetilde{\ca}_1$ in 
$\U(\hat{\gt q})/(\delta-c)$. Here $\U(\hat{\gt q})/(\delta-c)\cong\U_{\delta}(\hat{\gt q})$.

According to Lemma~\ref{lm-u-h}, 
$\bb(\hat{\gt q})=\bb(\gt q)-\dim\gt h+1$ and $\U_\delta(\hat{\gt q})=(\U(\gt q)\otimes_{\U(\gt h)}\mF)^H$. Now we consider  $\ca_1$ as an subalgebra of $(\U(\gt q)\otimes_{\U(\gt h)}\mF)^H$. 
After multiplying the elements of $\ca_1$ by suitable elements of $\mF$, we may safely assume that 
$\ca_1\subset \U(\gt q)^H$. Let $\ca=\mathsf{alg}\!\left<\ca_1,\gt h\right>\subset \U(\gt q)$
be the algebra generated by $\ca_1$ and $\gt h$. Then $\ca$ is commutative
and $\trdeg\!_{\mK}\,\ca = \trdeg\!_{\mF}\,\ca_1  + \dim_{\mK} \gt h  = \bb(\gt q)$.

$\bullet$ \  Suppose now that $\gt n$  contains no commutative characteristic ideals $\gt h$ such that $\dim\gt h>1$ or $[\gt q,\gt h]\ne 0$. Then either $\dim\gt n=1$ or
$\gt n$ is a Heisenberg Lie algebra, see Lemma~\ref{lm-h}. Let $\gt z\subset \gt n$ be the centre 
of $\gt n$. Since $\gt z$ is an Abelian ideal of $\gt q$, we have $[\gt q,\gt z]=0$. 
We will treat the cases of an algebraic and a non-algebraic $\gt q$ separately. 
For both of them, let $z\in\gt z$ be a non-zero element. 

$\bullet$ \ Consider first the algebraic case, where $\gt q=\gt l\ltimes\gt n$. 
If 
$\gt n=\gt z$, then $\gt q$ is a sum of two ideals. 
This contradicts our assumption 
on $\gt q$. Thus, 
$\gt n$ is a Heisenberg Lie algebra such that  $\dim\gt n\ge 3$ and $[\gt q,\gt z]=0$. 
By Corollary~\ref{cl-heis}, $(\U(\gt q)[z^{-1}])^{\gt n}\cong\U(\gt l)\otimes_{\mK}\mK[z,z^{-1}]$. 
Since $\gt l$ is reductive, there is  a quantum Mishchenko--Fomenko subalgebra
$\ca_\gamma\subset\U(\gt l)$ with $\trdeg\ca_\gamma=\bb(\gt l)$. Let $\ca_1$ be the image of this subalgebra 
in  $(\U(\gt q)[z^{-1}])^{\gt n}$. 
After multiplying the elements of $\ca_1$ by suitable powers of $z$, we may safely assume that 
$\ca_1\subset \U(\gt q)^{\gt n}$. 
Set $\ca=\mathsf{alg}\!\left<\ca_1,x_1,\ldots,x_n,z\right>\subset \U(\gt q)$ in the notation of Section~\ref{sec-i2}. 
Then $\ca$ is commutative and $\trdeg\ca=\bb(\gt l)+n+1$.  In view of \eqref{eq-bb-h}, 
$\trdeg\ca=\bb(\gt q)$. 

$\bullet$ \ Finally let $\gt q$ be a non-algebraic Lie algebra. 
Then $\gt q=\gt s\ltimes\gt r$. If $\gt z=\gt n$, then $[\gt r,\gt n]=[\gt r,\gt z]=0$ and $\gt r$ is the nilpotent 
radical of $\gt q$. 
Since $[\gt q,\gt z]=0$, we have also $\gt q=\gt s\oplus\gt z$, which contradicts our assumption 
on $\gt q$. Hence $\dim\gt n\ge 3$. 

According to Corollary~\ref{heis-non}, 
$(\U(\gt q)[z^{-1}])^{\gt n} \cong\U(\widetilde{\gt l})[z^{-1}]$.  
This isomorphism implies that 
$\ind \widetilde{\gt l} =\ind \gt q$. Note that 
$\dim(\widetilde{\gt l})=\dim\gt q-\dim\gt n+1$. 
By the inductive hypothesis, $\U(\widetilde{\gt l})$ contains a commutative subalgebra  
$\gc$ such that $\trdeg\gc=\bb(\widetilde{\gt l})$.
It produces  a commutative subalgebra $\ca_1\subset \U(\gt q)^{\gt n}$ of the same 
transcendence degree. 
The rest of the argument does not differ from the algebraic case above. 
\end{proof}

\section{Commutative subalgebras in subrings of invariants} \label{l-inv}

Let $\gt l\subset\gt q$ be a subalgebra. Then $\gr\!(\U(\gt q)^{\gt l})=\cS(\gt q)^{\gt l}$. 
Since  $\U(\gt q)^{\gt l}$ is a domain, combining Eq.~\eqref{eq-trdeg-l} with \cite[Satz~5.7]{bk-GK}, we obtain
\begin{equation} \label{trdeg-ul}
\trdeg\ca\le \bb(\gt q)-\bb(\gt l)+\ind\gt l=:\bb^{\gt l}(\gt q),
\end{equation} 
for any commutative subalgebra $\ca\subset\U(\gt q)^{\gt l}$. 
Note that if $\gt l$ is Abelian, then $\bb(\gt l)=\dim\gt l=\ind\gt l$ and hence 
$\bb^{\gt l}(\gt q)=\bb(\gt q)$. 

For $\gt l=\gt q$, we have $\bb^{\gt l}(\gt q)=\ind\gt q$. This shows already that
the upper bound cannot  be achieved in all cases. There are Lie algebras such that 
$\U(\gt q)^{\gt q}=\mK$ and $\ind\gt q\ge 1$.  An easy example is a Borel subalgebra
$\gt b\subset\gt{sl}_3$, where $\ind\gt b=1$.  

Sections~\ref{sec-i1} and \ref{sec-i2} combined with the proof of
Theorem~\ref{thm-U} show that there are two positive and very useful cases. Namely, 
if $\gt l$ is either an Abelian ideal of $\gt q$ consisting of $\ad$-nilpotent elements or a normal  Heisenberg subalgebra such that 
$[\gt l,\gt l]$ lies in the centre of $\gt q$, then $\U(\gt q)^{\gt l}$ contains a commutative 
subalgebra $\ca$ such that $\trdeg\ca=\bb^{\gt l}(\gt q)$. 

Therefore it stands to reason to look for appropriate classes of pairs $(\gt q,\gt l)$. 
We will concentrate on the case, where  $\gt q=\gt g$ is reductive. 
The study of  $\U(\gt g)^{\gt l}$ is motivated by the application to the branching 
rules $\gt g\downarrow \gt l$.

\begin{spec}
Take $\gt g=\gt{gl}_{2n}$. Then $\gt g$ contains a commutative subalgebra $\gt l$ of dimension $n^2$.  
For example, $\gt l=\left<E_{ij} \mid i\le n, j>n\right>_{\mK}$. Note that $\bb(\gt g)=2n^2+n$. 
To the best of my knowledge, no one ever looked at commutative 
subalgebras of $\U(\gt g)^{\gt l}$ or their Poisson-commutative counterparts.     
That could  be an interesting class of commutative subalgebras of $\U(\gt g)$ of the maximal possible transcendence degree. If contrary to my expectations, $\U(\gt g)^{\gt l}$ does not contain a 
commutative subalgebra of the transcendence degree $2n^2+n$,  then any maximal commutative 
subalgebra of  $\U(\gt g)^{\gt l}$ would provide an example of  a maximal w.r.t. inclusion commutative subalgebra that does not have the maximal possible transcendence degree. 
\end{spec}

\subsection{Centralisers} \label{sub-cent}

Consider the case $\gt l=\gt q_\gamma$ with $\gamma\in\gt q^*$. Here 
$\bar\ca_\gamma\subset\cS(\gt q)^{\gt l}$. If a reasonable quantisation 
exists, it has to lie in $\U(\gt q)^{\gt l}$. This is indeed the case 
for the quantum MF-subalgebra $\ca_\gamma\subset\U(\gt g)$ of a reductive 
Lie algebra $\gt g$, cf. Eq.~\eqref{ff-map}. Moreover, $\trdeg\ca_\gamma=\bb^{\gt l}(\gt g)$, see \cite[Lemma~2.1\,\&\,Prop.~4.1]{m-y}. 
If $\gamma\in\gt g^*_{\sf sing}$ and $\gamma$ is semisimple, then $\gt l=\gt g_\gamma$ is a proper Levi subalgebra 
of $\gt g$. The importance of $\ca_\gamma$ in the description of the branching rule 
$\gt g\downarrow \gt l$ is discovered in \cite{cris}.

\subsection{Symmetric subalgebras} \label{sub-sym} 
Suppose now that $\gt l=\gt g_0=\gt g^\sigma$, where $\sigma$ is an involution of $\gt g$. 
Poisson-commutative subalgebras $\eus Z\subset\cS(\gt g)^{\gt l}$ such that 
$\trdeg\eus Z=\bb^{\gt l}(\gt g)$ are constructed in \cite{oy}. Unfortunately, no quantisation of those 
subalgebras is known in general. 

\begin{ex} Take $\gt g=\gt{so}_{n+1}$, $\gt l=\gt{so}_n$. Then 
$\U(\gt g)^{\gt l}$ is commutative and is  generated by the centres $\mathcal{ZU}(\gt g)$, $\mathcal{ZU}(\gt l)$.   
Furthermore, $\bb^{\gt l}(\gt g)=\trdeg\U(\gt g)^{\gt l}$. 
\end{ex}

\section{On the notion of maximality} 

Suppose that $\ca\subset\U(\gt q)$ is a commutative subalgebra such that 
$\trdeg\ca=\bb(\gt q)$. It does not have to be {\it maximal w.r.t. inclusion}, but it is not far from it. 
Assume that $\ca\subset\gc\subset\U(\gt q)$ and $[\gc,\gc]=0$, then 
each element of $\gc$ is algebraic over $\ca$ and $\gc\subset Z_{\U(\gt q)}(\ca)$.  

\begin{prop} \label{prop-max}
Let $\ca$ be as above. Then $Z_{\U(\gt q)}(\ca)$ is a maximal commutative subalgebra of 
$\U(\gt q)$. With obvious changes the statement holds  for commutative 
subalgebras $\ca\subset\U(\gt q)^{\gt l}$. 
\end{prop}
\begin{proof}
Set $\gc=Z_{\U(\gt q)}(\ca)$. Since $\gc$ is an algebraic extension of $\ca$, it is commutative 
according to Corollary~\ref{cl-cl}. If $[\gc, x]=0$ for some $x\in \U(\gt q)$, then also 
$[\ca, x]=0$ and $x\in Z_{\U(\gt q)}(\ca)$. Hence $\gc$ is maximal. 
\end{proof}

If such an  $\ca$ is algebraically closed in $\U(\gt q)$, then it is maximal. Also if $\gr\!(\ca)$ is a 
maximal Poisson-commutative subalgebra of $\cS(\gt q)$, then 
$\ca$ is maximal. 
Both properties hold 
for the quantum 
Mishchenko--Fomenko subalgebras $\ca_\gamma\subset\cS(\gt g)$
with $\gamma\in\gt g^*_{\sf reg}$ 
\cite{codim3}. 
According to \cite{m-y}, $\ca_\gamma$ is a maximal commutative subalgebra  of $\U(\gt g)^{\gt g_\gamma}$ for 
any $\gamma\in\gt g^*$ if $\gt g$ is of type ${\sf A}$ or ${\sf C}$. 

The inductive steps in the proof of Theorem~\ref{thm-U} involve localisation. 
Therefore it is difficult to check, whether the  constructed subalgebras are maximal or not.  

\begin{ex}
Consider an easy example of a semi-direct product $\gt q=\gt l\ltimes\gt h$ with a Heisenberg Lie algebra.
Take $\gt l=\gt{sl}_2$ with a standard basis $\{e,h,f\}$ and 
$\gt h=\left<x,y,z\right>_{\mK}$. Suppose that $[e,y]=x$, $[e,x]=0$, $[f,x]=y$, $[f,y]=0$. 
Then over  $\mK[z,z^{-1}]$ the 
$\gt h$-invariants $\U(\gt q)^{\gt h}[z^{-1}]$ are generated by 
$$
zh + xy, \enskip 2ez - x^2, \enskip 2fz+y^2.
$$
Furthermore, $\cS(\gt q)^{\gt q}$ is generated by $z$ and 
$$
H_2=z(h^2+4ef)+2(hxy-fx^2+ey^2). 
$$
Identify $\gt{sl}_2\cong\gt{sl}_2^*$. Then we can take   the quantum MF-subalgebra of $\U(\gt{sl}_2)$
associated with either $h$ or $e$. In both cases, we pass to $\U(\gt q)^{\gt h}$ and 
add $x$ and $z$ as prescribed by the 
 proof of Theorem~\ref{thm-U}. 

The first algebra $\ca$ is 
$\mK[z,x,zh + xy,\Sym(H_2)]$. 
Calculations in the centraliser $\U(\gt q)^{\mK x}$ show that this one is maximal. 

The second algebra $\ca$ is 
different:
$$
\ca=\mK[z,x,2ez-x^2,\Sym(H_2)] \subset \mK[z,x,e,\Sym(H_2)].  
$$
It is not maximal.
\end{ex}

In 
the Abelian reduction step, we obtain $\ca=\mathsf{alg}\!\left<\ca_1,\gt h\right>$. 
If $\ca\subset\gc\subset\U(\gt q)$ and $\gc$ is commutative, then clearly 
$\gc\subset\U(\gt q)^{\gt h}$. However, some complications 
related to  denominators may appear here as well.

\end{document}